\documentclass[12pt,reqno]{amsart}
\usepackage[headings]{fullpage}
\usepackage{amssymb,amsmath,mathtools,bbm,tikz,tikz-cd,color,relsize,multirow}
\usepackage{stmaryrd}
\usepackage[all,cmtip]{xy}
\usepackage{url}

\usepackage{float}
\usetikzlibrary{positioning}
\usetikzlibrary{calc}
\usetikzlibrary{decorations.markings}
\usetikzlibrary{arrows}
\usetikzlibrary{calc}
\usetikzlibrary{decorations.markings}
\tikzstyle{hvector}=[inner sep=2pt,draw=blue!50,fill=blue!10,thick]
\tikzstyle{unit}=[inner sep=2pt,shape=circle, draw]
\tikzstyle{counit}=[inner sep=2pt,shape=circle, draw,fill=gray]
\tikzstyle{antipode}=[inner sep=2pt,shape=rectangle, draw]
\tikzstyle{cocycle}=[inner sep=2pt,shape=circle, draw]
\tikzstyle{twistedm}=[inner sep=2pt,shape=circle, fill=gray]
\tikzstyle{autom}=[inner sep=2pt,shape=circle, draw]
\tikzstyle{coact}=[inner sep=2pt,shape=circle, fill=black]

\tikzset{> /.tip = {Stealth[round,length=6pt]}}
\tikzstyle over=[preaction={draw,line width=5pt,white}]

\allowdisplaybreaks

\usepackage[bookmarks=true,%
    colorlinks=true,%
    linkcolor=blue,%
    citecolor=blue,%
    filecolor=blue,%
    menucolor=blue,%
    urlcolor=blue,%
    breaklinks=true]{hyperref}
\usepackage{slashed}    
\usepackage{verbatim}
\usepackage[normalem]{ulem}  
\usepackage{diagbox}

\newtheorem{theorem}{Theorem}[section]
\theoremstyle{definition}
\newtheorem{proposition}[theorem]{Proposition}
\newtheorem{lemma}[theorem]{Lemma}

\newtheorem{remark}[theorem]{Remark}
\newtheorem{corollary}[theorem]{Corollary}

\newtheorem{question}[theorem]{Question}

\def\red#1{\textcolor{red}{[#1]}}

\def\BZ{\mathbbm Z}
\def\BQ{\mathbbm Q}

\def\BC{\mathbbm C}

\def\a{\alpha}
\def\b{\beta}

\def\d{\delta}
\def\e{\epsilon}
\def\ve{\varepsilon}

\def\be{\begin{equation}}
\def\ee{\end{equation}}

\def\id{\mathrm{id}}

\def\c{\mathbbm{c}}

\def\sl21{\mathfrak{sl}(2|1)}

\def\fsl{\mathfrak{sl}}
\def\dim{\mathrm{dim}}

\newcommand{\lcoev}{\stackrel{\longleftarrow}{\operatorname{coev}}}
\newcommand{\lev}{\stackrel{\longleftarrow}{\operatorname{ev}}}
\newcommand{\rev}{\stackrel{\longrightarrow}{\operatorname{ev}}}

\def\phip{\varphi_{\ov{1}}^+(\mathfrak{g})}
\def\osp{\mathfrak{osp}}

\def\Q{\mathbb{Q}}
\def\C{\mathbb{C}}
\def\N{\mathbb{N}}
\def\Z{\mathbb{Z}}
\def\R{\mathbb{R}}
\def\sb{\subset}

\def\ot{\otimes}
\def\t{\times}

\def\c{\gamma}
\def\a{\alpha}
\def\b{\beta}
\def\d{\delta}
\def\e{\epsilon}
\def\De{\Delta}

\def\ov{\overline}

\def\wt{\widetilde}

\def\gg{\mathfrak{g}}
\def\id{\text{id}}
\def\sl{\mathfrak{sl}}

\def\rev{\overrightarrow{\text{ev}}}

\def\lev{\overleftarrow{\text{ev}}}
\def\lcoev{\overleftarrow{\text{coev}}}
\def\CC{\mathcal{C}}
\def\gg{\mathfrak{g}}
\def\coev{\text{coev}}

\def\TS{T_{\Sigma}}
\def\nn{\bold{n}}

\def\Vnna{V(\nn,a)}
\def\osp{\mathfrak{osp}}
\def\qi{q_i}

\def\tcb{\textcolor{blue}}
\def\Deev{\Phi^+_{\ov{0}}}
\def\Deodd{\Phi^+_{\ov{1}}}
\def\Uqg{U_q(\gg)}
\def\nn{\boldsymbol{n}}
\def\BB{\mathcal{B}}
\def\FF{\overline{F}}
\def\rhoev{\rho_{\overline{0}}}
\def\rhoodd{\rho_{\overline{1}}}
\def\o{\omega}
\def\S{\Sigma}
\def\Bnna{\mathcal{B}_{\nn,a}}
\def\HH{\mathcal{H}}
\def\CC{\mathcal{C}}
\def\Uqp{U_q(\mathfrak{p}^+)}
\def\Cnn{C_{\nn}}
\def\psl{\mathfrak{psl}}
\def\CC{\mathcal{C}_q^{w}(\gg)}
\def\ov{\overline}
\def\ggodd{\gg_{\ov{1}}}
\def\ggev{\gg_{\ov{0}}}
\def\sp{\mathfrak{sp}}
\def\UQ{U_{\mathbb{Q}}}
\def\wt{\widetilde}

\begin{document}

\title[Genus bounds for knot polynomials of type I Lie superalgebras]{
Genus bounds for knot polynomials of Lie superalgebras}

\author[S. Garoufalidis]{Stavros Garoufalidis}
\address{
  International Center for Mathematics, Department of Mathematics \\
  Southern University of Science and Technology \\
  Shenzhen, China \newline
  {\tt \url{http://people.mpim-bonn.mpg.de/stavros}}}
\email{stavros@mpim-bonn.mpg.de}

\author[D. López Neumann]{Daniel López Neumann}
\address{Universidad de Concepción, Concepción, Chile}
\email{dlopezn@udec.cl}


\date{8 July 2026}

\begin{abstract}
  Knot polynomials colored by typical representations of Lie superalgebras of type
  I (except $\psl(n|n)$) have two variables $q$ and $t$, the latter corresponding 
  to the complex-valued weight of the distinguished odd root. We prove that for every
  typical representation of a Lie superalgebra of type I, the $t$-degree of the 
  knot polynomial is at most the number of odd roots times the genus of the knot.
  A complimentary bound being at least the number of odd roots
  times degree of the Alexander polynomial can be obtained from a specialization
  at $q=1$. These two bounds become equalities when the Alexander polynomial detects
  the genus of the knot, as is the case for alternating knots and fibered knots.
\end{abstract}

\maketitle

{\footnotesize
\tableofcontents
}


\section{Introduction}
\label{sec.intro}

\subsection{Motivation}

The discovery of the Jones polynomial~\cite{Jones} gave rise to a plethora of
knot and link polynomials whose components were colored by representations of
Lie algebras and Lie superalgebras~\cite{RT:ribbon,Tu:YB}. Whereas Lie algebras
are a classic subject in mathematics, Lie superalgebras are less so. Yet
they play an essential role in supersymmetry and in geometry (see e.g., ~\cite{Manin}),
and their representation theory has been studied
extensively by Kac, Serganova and collaborators~\cite{Kac:superalgebra,Serganova}.
They also play a role in categorification~\cite{BS,Khovanov}, as well as in
non-semisimple topological quantum field theories in 2+1 dimensions~\cite{CGP}. 
\medskip

In this paper we highlight a connection between the quantum
knot polynomials and 3-dimensional topology witnessed by Lie superalgebras, and 
sketch a proof from first principles. More precisely, the
connection involves quantum knot invariants colored by representations of Lie
superalgebras, and a 3-dimensional topology invariant, the genus of the knot,
i.e., the smallest genus of an oriented spanning (Seifert) surface. 
This connection stems from the fact that some Lie superalgebras, called of type I 
\cite{Kac:reps}, have ``continuous'' families of simple modules, a fact that gives 
rise to 2-variable knot polynomials, where one variable is the usual $q$-variable, 
and the second one $t$ comes from the continuous representation-parameter. 
\medskip

A motivating example is the Links--Gould 2-variable polynomial invariant of 
knots~\cite{Links-Gould} associated with the 4-dimensional representation 
$V(0,a)$ of the Lie superalgebra $\fsl(2|1)$ (where $a \in \BC$ is a 
generic complex number, i.e., not an integer).
It was shown by Kohli--Tahar that this polynomial satisfies the genus
bounds~\cite{Kohli-Tahar}
\be
\label{gbound}
\deg_t J^{\fsl(2|1)}_{K,V(0,a)}(q,t) \leq 4 \, \text{genus}(K), \qquad (t=q^{2a})
\ee
for every knot $K$ in 3-space. Using a cabling formula for the Links--Gould 
polynomial and further arguments, it was shown in~\cite{Ga:colored} that the 
bound~\eqref{gbound} extends and holds when $V(0,a)$ is replaced by any 
$4(n+1)$-dimensional representation $V(n,a)$ of $\fsl(2|1)$
\be
\label{gkbound}
\deg_t J^{\fsl(2|1)}_{K,V(n,a)}(q,t) \leq 4 \, \text{genus}(K), 
\qquad (n \geq 0) \,.
\ee
This bound covers the set of all (typical) simple modules of $\fsl(2|1)$
and, curiously, the right hand-side of the above inequality does not depend
on the representation. 

\subsection{Our results} 

An explanation of why this happens, and what such a statement might be for all 
simple Lie superalgebras, led us to the notion of type I Lie superalgebra and typical 
representations, and to the importance of the set of odd positive roots. 
In Kac's classification, these superalgebras are $\sl(m|n), \psl(n|n)$ 
and $\osp(2|2n)$ and, except for $\psl(n|n)$, they are precisely the ones for which the 
even part has the form $\gg_{\ov{0}}=\gg_0 \oplus \C$ where $\gg_0$ is a semisimple Lie 
algebra~\cite{Kac:superalgebra}. A typical module $V=\Vnna$ of $\gg$ is determined by 
a dominant weight $\nn$ of $\gg_0$ plus a continuous parameter $a\in \C$ corresponding
to $\C$~\cite{Kac:reps}. Let $\phip$ denote the number of odd positive roots of $\gg$. 
Set $t=q^{2a}$ and consider the knot invariant
$J_{K,V}^{\gg}(q,t)\in \Z[q^{\pm 1}, t^{\pm 1}]$.


\begin{theorem}
\label{thm.1}
Fix a Lie superalgebra $\gg$ of type I and a typical representation $V=V(\nn,a)$
of $\gg$. Then, for any knot $K$ in 3-space we have 
\be
\label{grbound}
\deg_t J^{\gg}_{K,V}(q,t) \leq 2\phip \,\cdot \text{genus}(K) \,.
\ee
\end{theorem}

Our proof of Theorem~\ref{thm.1}, which follows the lines of \cite{LNV:genus-unrolled}, 
is independent of the proofs of the prior work for $\fsl(2|1)$
(see Equation~\eqref{gbound}), and it is even simpler. 
\medskip

In the above theorem $q$ is generic, that is, not a root of unity.
Note that when $\gg$ is a simple Lie algebra and the quantum parameter $q$ is a root 
of unity, then $U_q(\gg)$ also has families of simple modules depending on continuous
parameters. The corresponding polynomial invariants are generalizations of the
Akutsu--Deguchi--Ohtsuki (in short, ADO) invariants~\cite{ADO}. The fact that such
quantum invariants have genus bounds is a phenomenon first found in
\cite{LNV:genus-bounds, LNV:genus-unrolled}, but in that case, the genus bound
depends on the Lie algebra and (linearly) on the order of the root of unity.
In contrast, the genus bound in Theorem~\ref{thm.1} depends on the Lie (super)algebra
and not on the representation.
\medskip

A complementary bound to~\eqref{grbound} which also depends on $\phip$ 
can be obtained by an argument as in the MMR Conjecture~\cite{B-NG}.

\begin{theorem}
\label{thm.MMR}
For all $\gg$, $V$ and $K$ as above, we have:
\be
\label{eq.MMR}
J^{\gg}_{K,V}(t,1) = \Delta_K(t)^{\phip} \,.
\ee
\end{theorem}
For the representations $V(\boldsymbol{0},a)$ of $\fsl(m|1)$ this was shown 
in~\cite{Kohli-Pa}. Combining both bounds, we obtain the following.

\begin{corollary}
For all $\gg$, $V$ and $K$ as above, we have:
\be
\label{abound}
\phip \deg_t \Delta_K(t) \leq \deg_t J^{\gg}_{K,V}(t,q) \leq 2\phip \, \text{genus}(K)
\,.
\ee
\end{corollary}
Hence the inequalities in~\eqref{grbound} and in~\eqref{abound} are equalities 
when
\be
\label{notloose}
\deg_t \Delta_K(t)=2 \, \text{genus}(K) \,,
\ee
as is the case of fibered or alternating knots. A detailed study of knots and
their associated colored polynomials for the 4, 8, 12 and 16-dimensional
representations of $\fsl(2|1)$ is given in~\cite{GL:patterns}.
\medskip

Theorem \ref{thm.1} highlights the importance of Lie superalgebras in low dimensional
topology, as conjecturally their quantum knot invariants (which are always Vassiliev
invariants) may detect the genus of a knot. What's more, the $t$-variable of
the quantum knot polynomials suggests a deeper and yet not-understood connection
with a categorification theory which would explain the reason for the above
inequality, as well as a potential equality. 
\medskip

Note that the inequality~\eqref{gbound} is not an equality, since for instance the
left hand side is a mutation invariant, whereas the right hand side is not. Nevertheless,
it was experimentally observed in~\cite{GL:patterns} that the inequality~\eqref{gkbound}
is in fact an equality for $n=2$ and all 352.2 million knots with at most 
19 crossings. This motivates the following question.

\begin{question}
\label{que.1}
Fix a simple Lie superalgebra $\gg$ of type I other than $\psl(n|n)$. Is it true that
for every knot $K$ in 3-space, we have:
\be
\label{gebound}
\max_{V} \deg_t J^{\gg}_{K,V}(t,q) = 2\phip \, \text{genus}(K) 
\ee
where the maximum is taken over the set of all typical representations of $\gg$?
\end{question}

The proof of Theorem~\ref{thm.1} follows from first principles of quantum group theory
and topology, following the ideas of \cite{LNV:genus-unrolled, LNV:genus-bounds}. Explicitly, 

\noindent
$\bullet$ We fix a Seifert surface $\S$
of a knot $K$, put it in band form, and then fuse representations in each band to compute
the invariant $J^{\gg}_{K,V}$ in terms of the operator invariant $F(\TS,X)$ of a bottom tangle $\TS$ colored with $X=\Vnna\ot \Vnna^*$ and
a tensor attached to the base of the bands of the Seifert surface (this is the only common step of our proof with that of \cite{Kohli-Tahar}).

\noindent
$\bullet$ Following an idea of Habiro ~\cite{Habiro:bottom}, we compute the operator invariant $F(\TS,X)$ from the action of a universal invariant $Z_{\TS}\in U_q(\gg)^{\ot N}$ acting on $X^{\ot N}$, where $N=2g(\S)$.

\noindent
$\bullet$ 
A careful argument involving the Hopf algebra structure of $U_q(\gg)$ and PBW basis
implies a key algebraic bound for the action of $U_q(\gg)$ on $X$
(see Proposition \ref{prop.algb}). Namely, every element of $U_q(\gg)$ acts on the
``standard'' basis of $X$ with coefficients which are polynomials in $q^{-2a}$ of
degree at most the number of odd roots.

\noindent
$\bullet$ 
Applying this algebraic bound ($N$ times) to the action of $Z_{\TS}$ on $X^{\ot N}$ gives a proof of Theorem~\ref{thm.1}.
\medskip

Theorem~\ref{thm.MMR} follows from a MMR argument, which uses the fact that type I
Lie superalgebras have a nondegenerate symmetric bilinear form given by a supertrace
(analogue of the Killing form for simple Lie algebras) and 
identifies the specialization at $q=1$ with a product over all positive roots.
By the root properties of Lie superalgebras, each root contributes a factor of
$\Delta(t)$ or $1$ depending on being odd or even. This proves 
Equation~\eqref{eq.MMR}. 

\subsection{Further directions}
\label{sub.further}

We end this section by describing some further directions. For simplicity and elegance,
we stated Theorem~\ref{thm.1} for knots in $S^3$. However, its proof reveals that it 
also holds for links, once the appropriate definitions are discussed. To do so, 
we use any compact, oriented, possibly disconnected surface $\S$ whose boundary is an 
oriented link $L$ in $S^3$.
Then, the following analogue of inequality \eqref{grbound} holds. 

\begin{theorem}
\label{thm.1.links}    
Fix a Lie superalgebra $\gg$ of type I and a typical
representation $V$ of $\gg$. Then for every oriented link $L$ (with all components
colored by $V$) and every surface $\Sigma$ whose boundary is $L$, we have
\be
\label{grboundL}
\deg_t J^{\gg}_{L,V}(t,q) \leq \phip \, (1-\chi(\Sigma)) \,.
\ee
\end{theorem}

When $\S$ is connected, the proof is the same as in Theorem \ref{thm.1}, but for
disconnected $\S$ we need the symmetry property of the $J^{\gg}_{L,V}$'s of
Proposition \ref{prop: symmetry property}. As an example of Theorem \ref{thm.1.links},
consider the $(n,0)$-parallel $K^{(n,0)}$ of a 0-framed knot $K$ in $S^3$ and the
$\fsl(2|1)$-polynomial colored by the $V(0,a)$-representation (i.e., the Links--Gould
polynomial). Then, using $n$ copies of a minimal genus Seifert surface for $K$,
the above theorem asserts that
\be
\label{Knbounds}
\deg_t J^{\fsl(2|1)}_{K^{(n,0)}}(t,q) \leq 4ng -2(n-1) \,.
\ee
Using the formulas for $J^{\fsl(2|1)}_{K^{(n,0)}}(t,q)$ in terms of the
$V(0,ka)$-polynomials for $k \leq n$~\cite{Ga:colored} and the values of the latter
polynomials from~\cite{GL:patterns} we can indeed confirm the bounds~\eqref{Knbounds}
for $n=1,2,3,4$ for all knots with at most 11 crossings, and for $n=1,2$ for all
knots with at most 16 crossings.
\medskip

Regarding the analogue of Theorem~\ref{thm.1} when $q$ is a root of unity (i.e., the
ADO analogue for $\fsl(2|1)$), these link invariants have been studied by Ha~\cite{Ha}.
If one can lift the colored $\gg$ invariants of a knot for a type I Lie superalgebra
to a suitable Habiro ring of an \'etale map (as was described and conjectured
in~\cite[Rem.1.12]{GW:lift}), then the genus bounds for the knot polynomials for
generic $q$ imply (and are imply by) genus bounds for the corresponding polynomials
at roots of unity. Since this involves quite different ideas from those of the
current paper, we plan to discuss it in detail in a subsequent publication. 
\medskip

Concerning the computation of the $\mathfrak{sl}(m|n)$ knot polynomials colored by
the $2^{mn}$-dimension-al typical representation $V(0,a)$, this has been achieved
for $m=2,3,4,5$ and $n=1$ by Shana Li and the first author. The results of these
knot polynomials (which are mutation-invariant) will be discussed elsewhere.

\subsection*{Acknowledgements}
The authors wish to thank Nathan Geer and Vera Serganova for enlightening
conversations, as well as Ben-Michael Kohli and Matthew Harper for comments
on a first draft of the paper.


\section{Classical and quantum Lie superalgebras}
\label{sec.lie}

In this section we recall the basic definitions of Lie superalgebras and their
quantized versions.

\subsection{Basics on Lie superalgebras}
\label{sub.basics}

A (classical) Lie superalgebra is a (finite dimensional) $\BZ/2\BZ$-graded vector space
$\gg=\gg_{\overline 0} \, \oplus \, \gg_{\overline 1}$ of even and odd subspaces with
a graded Lie bracket that satisfies a graded version of the antisymmetry and the Jacobi
identity. Simple Lie superalgebras (i.e. having no non-trivial ideals) were classified by 
Kac \cite{Kac:superalgebra}, they are divided into classical ones (i.e. $\gg_{\ov{1}}$ is 
completely reducible as a $\gg_{\ov{0}}$-module) and non-classical (also known as
Cartan type). In turn, classical simple Lie superalgebras are divided into type I and
type II, depending on whether $\ggodd$ is a reducible $\ggev$-module or not.

\medskip
In this paper we are interested in simple Lie superalgebras in which the even part
splits as $\gg_{\ov{0}}=\gg_0\oplus \C$ where $\gg_0$ is a semisimple Lie algebra.
The only such algebras are the special linear $\sl(m|n)$ with $m\neq n$ and the
orthosymplectic $\osp(2|2n)$ \cite{Kac:superalgebra}. These have even parts 
$$
\sl(m|n)_{\ov{0}}=\sl(m)\oplus \sl(n)\oplus \C, \qquad
\osp(2|2n)_{\ov{0}}= \sp(2n)\oplus \C.
$$
Note that $\sl(n|n)$ has non-trivial center $Z$ (of dimension 1 generated by the
identity matrix), hence it is not simple, but the quotient $\psl(n|n):=\sl(n|n)/Z$ is
simple. According to Kac, $\sl(m|n)$ with $m\neq n$, $\psl(n|n)$ and $\osp(2|2n)$
account for all type I Lie superalgebras. However, $\psl(n|n)_{\ov{0}}$ is semisimple,
hence the theorems of this paper do not apply.

\subsection{Roots}
\label{sub.roots}

The set $\Phi^+=\Deev\cup\Deodd$ of positive roots of a Lie superalgebra 
is a disjoint union of the sets of even  $\Deev$ and odd roots $\Deodd$ respectively.

An important feature of type I Lie superalgebras is that 
\begin{itemize}
\item
  there is a unique odd simple root $\a_m$ (we denote $\a_m=\a_1$ for
  $\mathfrak{osp}(2|2n)$), 
\item
  every odd root $\a \in \Deodd$ can be written uniquely in the form $\a=\b+\a_m$
  for some $\b \in \Deev$.
\end{itemize}

We discuss roots on $\fsl(m|n)$ first, which has rank $r=m+n-1$.
Its positive roots are described  in terms of a vector space with orthogonal basis
$(\ve_1,\dots,\ve_m,\delta_1,\dots,\delta_n)$
and inner product
\be
\label{islmn}
(\ve_i,\ve_j)= \delta_{ij}, \qquad (\delta_i, \delta_j)=-\delta_{ij}, \qquad
(\ve_i,\delta_j)=0 \,.
\ee
The simple roots are given by 
\be
\begin{aligned}
\Pi =& \{\a_1=\ve_1-\ve_2,\dots, \a_{m-1}=\ve_{m-1}-\ve_m, \a_m=\ve_m-\delta_1, 
\\ &
\a_{m+1}= \delta_1-\delta_2, \dots, a_{m+n-1}=\delta_{m+n-1}-\delta_{m+n} \} 
\end{aligned}
\ee
which are all even except $\a_m$ which is odd.
The symmetrized Cartan matrix of $\fsl(m|n)$ is the $r \times r$ matrix of inner
products $(\a_i,\a_j)$. For the unsymmetrized Cartan matrix $(a_{ij})$, see
\cite[page 100]{Frappat}. 
Setting $d_i=1$ for $1\leq i\leq m$ and $d_i=-1$ for $i=m,\dots,r$, the symmetrized and 
unsymmetrized Cartan matrices are related by $(\a_i,\a_j)=a_{ij}d_i$.
The positive roots are given by $\a_{ij}$ for $1 \leq i < j \leq m+n$, and
\begin{itemize}
\item
  the even ones are those with $1 \leq i < j \leq m$ or $m+1 \leq i < j \leq m+n$ 
\item
  the odd ones are those with $1 \leq i \leq m < j \leq m+n$. 
\end{itemize}
A convex order of the set of positive roots is the lexicographic order
$\a_{ij}<\a_{k\ell}$ if and only if $i<k$ or $i=k$ and $j < \ell$.
\medskip

Next, we discuss the Lie superalgebra $\osp(2|2n)$ which has rank $n+1$. Its
positive roots are described in terms of 
a vector space with orthogonal basis $(\ve_1,\delta_1,\dots,\delta_n)$
and inner product
\be
\label{iosp2n}
(\ve_1,\ve_1)=1, \qquad (\delta_i, \delta_j)=-\delta_{ij}, \qquad
(\ve_1,\delta_i)=0 \,.
\ee
The simple positive roots are given by
\be
\Pi = \{\a_1=\e_1-\d_1,\a_2=\delta_1-\delta_2, \a_3=\delta_2-\delta_3, \dots, \a_{n}=
\delta_{n-1}-\delta_n, \a_{n+1}=2\delta_1 \} \,
\ee
which are all even except $\a_1$ which is odd. The symmetrized Cartan matrix of 
$\osp(2|2n)$ is the $(n+1) \times (n+1)$ 
matrix $(\a_i,\a_j)$, this is related to the unsymmetrized Cartan $(a_{ij})$ 
by $(\a_i,\a_j)=a_{ij}d_i$ where $d_1=1, d_i=-1$ for $2\leq i\leq n$ and $d_{n+1}=-2$.
The Lie superalgebra $\osp(2|2n)$ has $n^2$ even positive roots
\begin{itemize}
\item
  $2 \delta_i$ for $1 \leq i \leq n$
\item
  $\delta_i \pm \delta_j$ for $1 \leq i < j \leq n$
\end{itemize}
and $2n$ odd positive roots $\ve_1 \pm \delta_k$ for $1 \leq k \leq n$.
A convex ordering of the set of positive roots is given by
\be
\begin{aligned}
  & \delta_1 -\delta_2 < \delta_1 -\delta_3 < \dots < \delta_1 -\delta_n \\
  < & \,\ve_1 -\delta_n < \ve_1 -\delta_{n-1} < \dots < \ve_1 -\delta_1 
  <  \,\ve_1 +\delta_1 < \ve_1 +\delta_{2} < \dots < \ve_1 +\delta_n \\
  < & \,\delta_n +\delta_{n-1} < \dots < \delta_2 -\delta_1 
  <  \, 2 \delta_n < 2 \delta_{n-1} < \dots < 2 \delta_1 \,. 
\end{aligned}  
\ee

\begin{remark}
In both cases ($\sl(m|n)$ and $\mathfrak{osp}(2|2n)$) we have $d_m=1$ 
(recall $m=1$ for $\mathfrak{osp}(2|2n)$).
\end{remark}

We denote by $L_R$ the root lattice, i.e. the lattice in the dual Cartan subalgebra 
generated by the simple roots. This is an inner product space with the pairing 
$(\a_i,\a_j)$ as above. 

\subsection{Quantum supergroups} 

Let $(a_{ij})$ be the (unsymmetrized) Cartan matrix of $\gg=\sl(m|n)$ 
(with $m\neq n$) or $\gg=\osp(2|2n)$. Let $q$ be a generic parameter and set 
$q_i=q^{d_i}$ for $i=1,\dots,r$ ($r=$ rank of $\gg$). Then the quantum supergroup
$U_q(\gg)$ is the $\C$-algebra with generators $E_i,F_i,K_i^{\pm 1}$, $i=1,\dots,r$ 
and relations
\begin{align*}
K_iE_j&=\qi^{a_{ij}}E_jK_i, & K_iF_j&=\qi^{-a_{ij}}F_jK_i & K_iK_j&=K_jK_i, \\
[E_i,F_j]&=\d_{ij}\frac{K_i-K_i^{-1}}{\qi-\qi^{-1}}, & K_iK_i^{-1}&=K_i^{-1}K_i=1, & & 
\end{align*}
for every $i,j$, where $[x,y]=xy-(-1)^{|x||y|}yx$ denotes the super-commutator,
plus $q$-Serre relations for which we refer the reader to \cite{Yamane}. 
We declare $E_m,F_m$ to be odd while all other generators are declared 
to be even (recall that we set $m=1$ for $\osp(2|2n)$).

\medskip
This is a Hopf superalgebra if we define 
\begin{align*}
\De(E_i)&=E_i\ot K_i+1\ot E_i, & \De(F_i)&=K_i^{-1}\ot F_i+F_i\ot 1,& 
 \De(K_i)&=K_i\ot K_i, 
\end{align*}
 and
\begin{align*}
S(E_i)&=-E_iK_i^{-1}, & S(F_i)&=-K_iF_i, & S(K_i)&=K_i^{-1}
\end{align*}
for every $i=1,\dots,r$. 
\medskip

In what follows, we denote by $\UQ^+$ (resp. $\UQ^-$) the $\Q(q)$-subalgebra of
$U_q(\gg)$ generated by $E_1,\dots,E_r$ (resp. $F_1,\dots, F_r$). We denote by
$\UQ^0$ the Cartan $\Q(q)$-subalgebra generated by the $K_1,\dots,K_r$.

\subsection{Root vectors} 

As mentioned above, the roots are $\Phi^+=\Deev\cup\Deodd$. Fixing a convex order for
$\Phi^+$, for every $\a\in \Phi^+$ one can define a root vector $E_{\a}\in\UQ^+$ as 
in ~\cite[Sec.5]{Yamane}. Then, for every function $J:\Phi^+\to \N_{\geq 0}$ we can
define 
$$
E^J:=\prod_{\a\in\Phi^+}E_{\a}^{J(\a)}\in \UQ^+.$$
Note that $E^J=0$ if $J(\a)\geq 2$ for some $\a\in\Deodd$. We will similarly define
vectors $F^J\in \UQ^-$ for every $J:\Phi^+\to \N_{\geq 0}$. 

\begin{lemma}
\label{lemma: odd root vectors have a unique Em}
\begin{enumerate}
\item 
For every $\a\in\Deev$ one can write 
$$
E_{\a}=\sum_I c_{I}\cdot E_{i_1}\dots E_{i_k}
$$
where each $E_{i_j}$ is even and $c_I\in \Q(q)$ for every $I=(i_1,\dots,i_k)$.
\item 
For every $\a\in\Deodd$ one can write 
$$
E_{\a}=\sum_I c_{I}\cdot E_{i_1}\dots E_{i_k}
$$
where for every $I=(i_1,\dots,i_k)$ there is a unique $j$ such that $E_{i_j}$ 
is odd (hence $i_j=m$).
\end{enumerate}
A similar statement holds for the negative root vectors.
\end{lemma}

\begin{proof}
An induction from Yamane's formulas.
\end{proof}

For every $J:\Phi^+\to \N_{\geq 0}$ we denote 
$$
d(J):=\sum_{\a\in \Phi^+}J(\a)\a\in L_R.
$$
Then, it follows that  
$$
K_iE^J=q^{(d(J),\a_i)}E^JK_i, \ \ K_iF^J=q^{-(d(J),\a_i)}F^JK_i
$$
for every $J:\Phi^+\to \N_{\geq 0}$ and every $i=1,\dots,r$.

\subsection{PBW bases}
\label{sub.PBW}

The universal enveloping algebra of a simple Lie (super) algebra has a PBW basis
that depends on any ordering of the set of positive roots of it. So does the quantized 
Lie (super) algebra, but this requires a definition of $E_\a$ for all positive roots
$\a$. Such a definition can be given inductively, using for instance a convex order
of the set of positive roots, that is, an order $<$ that satisfies the property
that if $\a,\b,\a+\b$ are positive roots with $\a<\b$, then $\a<\a+\b<\b$. That is,
the sum of two roots is in the interval between the two. For simple Lie algebras
convex orders are in bijection with all reduced expressions of the longest element
of the Weyl group. For Lie superalgebras, there is no Weyl group (though there is
a Weyl groupoid). On the other hand, there are rather canonical orderings of the
set of positive roots of the Lie superlagebras $\fsl(m|n)$ and $\osp(2|2n)$.

In~\cite[Sec.2.1]{Yamane} Yamane discusses an ordering of the positive roots of
$\fsl(m|n)$
and of $\osp(2|2n)$. Some of our arguments below involve $q$-commutation relations
known as the Levendorki--Soibelman formula~\cite{LS}, which uses
a convex order of the set of positive roots and expresses the $q$-commutator of
$E_\a$ with $E_\b$ for two positive roots $\a$ and $\b$ with $\a < \b$ as follows:
\be
\label{LS}
E_\a E_\b- q^{(\a,\b)} E_\b E_\a = \sum
c_{\gamma,n}(q) E_{\gamma_1}^{n_1} \dots E_{\gamma_k}^{n_k} 
\ee
where $c_{\gamma,n}(q) \in \BQ(q)$, and the sum involves tuples
$(\gamma_1,\dots,\gamma_k)$ and nonnegative integers
$(n_1,\dots,n_k) \in \BZ_{\geq 0}^k$
that satisfy
\be
\label{LScond}
\a < \gamma_1 < \dots < \gamma_k < \b, \qquad
\a+\b = n_1 \gamma_1 + \dots + n_k \gamma_k \,.
\ee
The above formula holds for quantum Lie superalgebras, see~\cite[Lem.5.2.1]{Yamane}.

For type I superalgebras, when $\a$ is even and $\b$ is odd (and either $\a<\b$ or
$\b <\a$), it follows that there is a unique $\gamma_i$ which is odd, and all
others are even. Moreover, since $E_{\text{odd}}^2=0$, we necessarily have $n_i=1$. 
In other words, we have
\be
\label{LS2}
E_\a E_\b- q^{(\a,\b)} E_\b E_\a = \sum
c_{\gamma,n}(q) E_{\gamma_1}^{n_1} \dots E_{\gamma_{i-1}}^{n_{i-1}} 
E_{\gamma_i} E_{\gamma_{i+1}}^{n_{i+1}} E_{\gamma_k}^{n_k}, \qquad 
(\gamma_i = \text{odd}, \,\, \gamma_{\neq i} = \text{even} )
\ee
with $\gamma_j$ still satisfying~\eqref{LScond}.
But in fact, using the classification of Lie superalgebras of type I, the above 
formula simplifies further. We discuss this explicitly.

We start with $\fsl(m|n)$. Recall its roots from Sectiion~\ref{sub.roots}.
Abbreviating $E_{\a_{ij}}$ by $E_{ij}$ suppose that $E_{ij}$ is even and $E_{kl}$
is odd. Then, the Levendorski--Soibelman formula of
the $q$-commutator of $E_{ij}$ and $E_{kl}$ takes the form: 

\be
\begin{aligned}
i < j = k \leq m < l & \,\, : \,\, E_{ij} E_{jl}-q^{-1} E_{jl} E_{ij} = E_{il} \\
k \leq m < l = i < j & \,\, : \,\, E_{ij} E_{ki} - q E_{ki} E_{ij} = -q E_{ki} \\
\text{otherwise} & \,\, : \,\, E_{ij} E_{kl} - q^{(\a_{ij},\a_{kl})} E_{kl} E_{ij} = 0 \,.
\end{aligned}
\ee

See also~\cite[Lemma, p.2555]{Hakobyan}. In other words,

\begin{corollary}
\label{cor.slmn}
The $q$-commutator of an even $E_\a$ with an odd $E_\b$ is either
0 or a multiple of the odd $E_{\a+\b}$.
\end{corollary}

Next, we discuss the Lie superalgebra $\osp(2|2n)$. Recall its roots from
Section~\ref{sub.roots}. The Levendorski--Soibelman formula of a $q$-commutation of 
$E_\a$ and $E_\b$ for even $\a$ and odd $\b$ takes the form: 

\begin{itemize}
\item[(1)]
  If $(\a,\b)=(\delta_i-\delta_j,\ve_1+\delta_i)$ 
  or $(\a,\b)=(\delta_i-\delta_j,\ve_1-\delta_j)$
  then
  $E_\a E_\b- q^{(\a,\b)} E_\b E_\a = E_{\b-\a}$.
\item[(2)]
  If $\a=\delta_i+\delta_j$ and $\b=\ve_1-\delta_j$, then
  $E_\a E_\b- q E_\b E_\a = E_{\a+\b}$.
\item[(3)]
 If $\a=2 \delta_j$ and $\b=\ve_1-\delta_j$, then
$E_\a E_\b- q^{2} E_\b E_\a = (q+q^{-1}) E_{\a+\b}$.
\item[(4)]
  In all other cases, 
  $E_\a E_\b- q^{(\a,\b)} E_\b E_\a = 0$.  
\end{itemize}

In other words, 

\begin{corollary}
\label{cor.osp}
The $q$-commutation of $E_\a$ and $E_\b$ for
even $\a$ and odd $\b$ is either zero, or a multiple of some $E_{\text{odd}}$. 
\end{corollary}

\subsection{Kac modules} 
\label{sub.kac}

As mentioned above, we care about the cases $\gg=\sl(m|n)$ ($m\neq n$) and
$\gg=\osp(2|2n)$ since $\gg_{\ov{0}}=\gg_0\oplus \C$ where $\gg_0$ is a semisimple
Lie algebra (in the usual sense). This fact is what permits to construct simple
modules of arbitrary complex highest weight as we now explain.
\medskip

Let $V$ be a finite-dimensional simple module over $U_q(\gg_0)$. Then we know that
$V$ is a highest weight module and the weights are all integers, let $v_0$ be a
highest weight vector. 
Moreover, $V$ has a basis of vectors of the form $F^Jv_0$ for some finite set of even
vectors $F^J$. The weights of $v_0$ are parametrized by a discrete parameter $\nn$,
so we denote $V=V(\nn)$. Given a parameter $a\in \C$ we make $V$ into a
$U_q(\gg_{\ov{0}})$-module by setting 
$$ 
K_mv_0=q^{a}v_0
$$
and extending $K_m$ to all of $V$ using $K_mF_j=q^{-a_{mj}}F_jK_m$ (recall that $d_m=1$ 
so $q_m=q$). Let $\Uqp$ be the subalgebra of $\Uqg$ generated by $U_q(\gg_{\ov{0}})$ and 
$E_m$. Now make $V$ into a $\Uqp$-module by setting 
$$
E_mv_0=0
$$
which forces $E_m=0$ over all of $V$ since $[E_m,F_i]=0$ for all even $F_i$. 
Then we obtain a $\Uqg$-module by setting 
$$
\Vnna=\Uqg\ot_{\Uqp}V(\nn).
$$

\begin{lemma}(\cite{Kac:reps})
\label{lem.simple}
When $(\nn,a)$ is typical (in particular if $a\in \C\setminus\Z$), the module
$\Vnna$ is simple.
\end{lemma}
It is well-known that $V(\nn)$ has a basis of the form $\{F^Bv_0\}_{B\in \Cnn}$ for
some set $\Cnn$ of functions $B:\Deev\to \N_{\geq 0}$.

\begin{lemma}
\label{lemma: PBW on V(n,a)}
The module $\Vnna$ has a PBW basis of the form $F^AF^Bv_0$, where $F^A$ is a product 
of odd roots (so $A:\Deodd\to\{0,1\}$) and $B\in \Cnn$.
\end{lemma}

\begin{proof}
First, we know $\dim(\Vnna)=2^{\phip} \dim(V(\nn))$. 
So, it suffices to show that $F^A F^B$ with $A$ and $B$ as above spans $\Vnna$.
For this, choose a convex order on $\Delta^+$. Then, $\prod_{\a \in \Delta^+} F_\a^{n_a}$ is a PBW basis of
$\UQ^-$ where $n_\a \in \BZ_{\geq 0}$ if $\a$ is even and $n_\a \in \{0,1\}$ if
$\a$ is odd. Moreover, $F_\a$ for even (resp., odd) $\a$ is a sum of a product 
of even $F_{\a_i}$ (resp., a product of even $F_{\a_i}$ and $F_m$ once). 
By the induced module definition, we get a PBW basis on $\Vnna$ (but not of the desired form). 

Let $W$ denote the subspace of $\Vnna$ spanned by $F^A F^B v_0$ for $A$ and $B$
as in the statement of the lemma. We claim that for every even root $\a$, we have
$F_\a W \subset W$. This follows by induction on $|A|$ and from the $q$-commutation
relations of Corollaries~\ref{cor.slmn} and~\ref{cor.osp}, which imply that
the $q$-commutator of $F_\a$ and $F_{\text{odd}}$ is either 0, or a multiple of
some $F_{\text{odd}}$, or a multiple of an even times some $F_{\text{odd}}$.

The above claim, together with a PBW basis of $\Vnna$ coming from a convex order
implies that $W=\Vnna$. It follows that $\Vnna$ is spanned by the vectors
$F^A F^B v_0$ which are as many as its dimension. Hence, the said vectors form
a basis of $\Vnna$.
\end{proof}

\begin{remark}
    In fact, any PBW basis of $\Vnna$ is good enough to make the argument of Proposition \ref{prop.algb} work, but we prefer to have a basis of the above form.
\end{remark}

\subsection{The braiding} 

By a weight module over $U_q(\gg)$ we mean a finite-dimensional module over which the
action of $\UQ^0$ is diagonalizable and eigenvalues of the $K_i$'s are all powers of
$q$. We denote by $\CC$ the category of weight modules over $U_q(\gg)$. 
\medskip

Denote by $F_{\Phi^+}$ the set of functions $I:\Phi^+\to \N_{\geq 0}$ 
such that $I(\a)\in \{0,1\}$ for every $\a\in \Deodd$. 

\begin{proposition}\cite{Yamane}
For any $U,V$ in $\CC$, the braiding $c_{U,V}:U\ot V\to V\ot U$ is given by 
\be
\label{cuv}
c_{U,V}(u\ot v)=\tau_{U,V}\left(\HH (\Theta \cdot (u\ot v))\right)
\ee
where 
\begin{itemize}
\item
$\tau_{U,V}(x\ot y)=y\ot x$ is the usual transposition, 
\item 
$\HH(x\ot y)=q^{(\b,\c)}x\ot y$ if $x\in U,y\in V$ have weights $\b,\c\in L_R$ 
respectively is defined by
\be
\label{Hpart}
\HH=q^{\sum_{i,j}a'_{ij}H_i\ot H_j}, \qquad 
\ee
where $(a'_{ij})$ denotes the inverse matrix of $(a_{ij}d_i)$ and where $H_i$ 
is defined by 
$K_i=q^{H_i}$ for each $i=1,\dots,r$,
\item  
and 
$$
\Theta=\sum_{I\in F_{\Phi^+}}c_I\cdot E^I\ot F^I
$$
is the quasi-R-matrix for some coefficients $c_I\in \Q(q)$.
\end{itemize}
\end{proposition}

Note that although $\Theta$ is an infinite sum, only finitely many terms act nonzero
in each of the weight modules considered in this paper.

\subsection{Pivotal and ribbon structure} 
\label{sub.piv}
The equation 
\be
S^2(E_i)=K_{\c}E_iK_{\c}^{-1}=q^{(\c,\a_i)}E_i = q^{(\a_i,\a_i)}E_i
\ee
for an element $\c\in L_R$ determines $\c$ uniquely since the Cartan matrix
is invertible. An explicit computation using the positive roots of $\fsl(m|n)$ 
and $\osp(2|2n)$ given in Section~\eqref{sub.PBW} reveals that 
\be
\label{rho}
\c = 2\rho, \qquad \rho=\rhoev-\rhoodd
\ee
satisfies the equation $(\c,\a_i)=(\a_i,\a_i)$ for all $i$ (see also~\cite[Lem.3.18]{AGP}
for $\fsl(m|n)$), where 
$$
2\rhoev=\sum_{\a\in\Deev}\a, \qquad 2\rhoodd=\sum_{\a\in\Deodd}\a.
$$
\medskip

Thanks to the pivotal structure, for every $V\in \CC$ there is an isomorphism 
\be
\label{jVdef}
j_V:V^{**}\to V, \qquad 
j_V^{-1}(v)(v')=v'(K_{2\rho}v).
\ee

The category $\CC$ is ribbon with the twists $\theta_V$ defined by
$$
\theta_V=(\id_V\ot\rev_V)(c_{V,V}\ot \id_{V^*})(\id_V\ot \lcoev_V).
$$
Here $\rev_V: V \ot V^* \to \BC$ is the right evaluation defined by 
$\rev_V(v\ot v')=v'(K_{2\rho}v)$ and $\lcoev_V: \BC\to  V \ot V^*$ is the 
left co-evaluation defined by $\lcoev(1)=\sum_i v_i \ot v_i^*$.

\medskip
For a module $V=V(\nn,a)$ the action of the twist can be computed from the above formula 
to obtain 
$$
\theta_{V}=q^{\sum_{i,j=1}^ra'_{ij}n_in_j+\sum_{i=1}^r k_in_i}\id_V
$$
where we set $n_m=a$ and the $k_i\in \Z, i=1,\dots, r$ are characterized by 
$2\rho=\sum_{i=1}^rk_i\a_i$. Note that this coefficient belongs to $\Q(q,q^a,q^{a^2})$.

\section{Link polynomials}

\subsection{Quantum invariants of links from $(1,1)$-tangles} 

We assume the reader is familiar with the Reshetikhin-Turaev invariants of framed, 
oriented tangles from ribbon tensor categories presented in~\cite{RT:ribbon,Tu:book}. 
If $T$ is a framed, oriented tangle, we denote by $F(T,V)$ the Reshetikhin-Turaev
invariant of 
$T$ in which every component is colored with a module $V$ in $\CC$. When discussing link
polynomials associated to representations of type I Lie superalgebras, there is a
modification of the Reshetikhin--Turaev functor, first introduced by Geer--Patureau 
in~\cite{Geer:multi1}, which amounts to cutting a component of a framed, oriented link, 
computing the invariant and then multiplying by a modified dimension (since the quantum 
dimension of $V(\nn,a)$ vanishes). Our invariants do not require modified
trace/dimension as we now explain. 
\medskip

Let $V=V(\nn,a)\in \CC$ for typical $(\nn,a)$. Let $\wt{L}$ be a framed, oriented link
in $S^3$ and let $\wt{L}_o$ be a $(1,1)$-tangle with closure $\wt{L}$. Then
$F(\wt{L}_o,V):V\to V$ must have the form $c_{\wt{L}_o}\cdot \id_V$ since $V$ is
simple, where $c_{\wt{L}_o}\in \Q(q,q^a,q^{a^2})$. 
It can be proved that $c_{\wt{L}_o}$ is an invariant of $\wt{L}$, that is, it is
independent of where the link was cut open. This follows from the results of
\cite{GPT:modified}: to get link invariants one typically needs to cut the link open
and multiply by the modified dimension of the color of the open component, but since
all components have the same color, this last step is unnecessary.
\medskip

Now, if $L$ is an unframed, oriented link, and $\lambda_{\nn,a}\in \Q(q,q^a,q^{a^2})$
is such that $\theta_{V(\nn,a)}=\lambda_{\nn,a}\cdot \id_{V(\nn,a)}$, then
$$
J^{\gg}_{L,V(\nn,a)}:=\lambda_{\nn,a}^{-w(\wt{L})}\cdot c_{\wt{L_o}}
$$
(where $c_{\wt{L}_o}$ was defined above) is an invariant of $L$ that belongs
to $\Q(q,q^a)$, and in fact, it is in $\Z[q^{\pm 1},q^{\pm a}]$. Here $\wt{L}$ denotes
$L$ with an arbitrary choice of framing and $w(\wt{L})$ denotes the writhe, that is,
the sum of signs over all crossings in a diagram of $\wt{L}$ (with blackboard framing).
Up to an overall power of $q^a$, this invariant is a Laurent polynomial in $q^{2a}$ so
we set $t=q^{2a}$ and denote the invariant by $J_{L,V}^{\gg}(q,t)$.
\medskip

\begin{remark}
\label{rem.int}
The proof of Theorem~\ref{thm.1} uses only the weaker property that
$J_{L,V}^{\gg}(q,t)$ lies in $\BQ(q)[q^{\pm a}]$, whereas it actually lies in the
ring $\BZ[q^{\pm 1},q^{\pm a}]$, as follows from adapting mutatis mutandis the
arguments of L\^{e}~\cite{Le:strong} from the case of simple Lie algebras to the
case of Lie superalgebras of type I.
\end{remark}

\begin{remark}
The invariants $J^{\gg}_{L,V(\nn,a)}$ are specializations of the multivariable
invariants of \cite{Geer:multi1} (after dividing by the modified dimension of
$V(\nn,a)$ and equating the variables corresponding to every component of the link).
\end{remark}

\subsection{Weyl symmetry} 

In this section we describe a Weyl-involution on the set of typical Kac modules,
which is used in the proof of Theorem~\ref{thm.1.links} only, and may be skipped for
the rest of the proofs. Let $\o:U_q(\gg)\to U_q(\gg)$ be the $\Q(q)$-linear
automorphism defined by
\be
\label{omegadef}
\o(E_i)=F_i, \qquad \o(F_i)=E_i, \qquad \o(K_i)=K_i^{-1} \,.
\ee
Then $\o(E_\a)=F_\a$ and $\o(F_\a)=E_\a$; see~\cite[Sec.4.6]{Jantzen}. 

Moreover, if $\nn$ is the highest weight of the module $V(\nn)$ of $\gg_0$, then
$\omega$ restricts to $U_q(\gg_0)$ and~\cite[Sec.5.16]{Jantzen}
\be
\label{oVn}
\o(V(\nn)) \simeq V(\nn)^* \simeq V(\nn') \,.
\ee
It follows by the induced map construction of the Kac module that
\be
\label{oVnn}
\o(V(\nn,a))\cong V(\nn',-a+l_{\nn})
\ee
for some integer $l_{\nn}\in \Z$.

\begin{proposition}(Symmetry)
\label{prop: symmetry property}
The link invariant $J_{L,V}^{\gg}$ satisfies the symmetry property 
$$
J^{\gg}_{L,V(\nn,a)}(t,q)=J^{\gg}_{L,V(\nn',a)}(q^{l_{\nn}}t^{-1},q).
$$
\end{proposition}

\begin{proof}
Compute $J^{\gg}_{L,V(\nn,a)}$ from a braid closure of $L$, opened at the leftmost
strand. Now rotate this open-braid by 180 degrees horizontally. The result is that overcrossings get
exchanged by undercrossings and vice-versa. Similarly, right caps become left caps
(similarly for cups). 
Thus, computing the link invariant $J^{\gg}_{L,V(\nn,a)}$ from this diagram is the same
as computing 
from the original one colored with $\o(V(\nn,a))$. But since these are link invariants, 
the result must be the same, that is $$J^{\gg}_{L,V(\nn,a)}=J^{\gg}_{L,\o(V(\nn,a))}$$
and using $\o(V(\nn,a))\cong V(\nn',-a+l)$ the result follows.
\end{proof}

Moreover, since $\o(V(\nn,a))\cong V(\nn,a)^*$ and taking duals corresponds to reversing
orientation, the above proof shows that $J^{\gg}_{L,V(\nn,a)}$ is invariant under an
overall orientation reversal.

\begin{remark}
For $\sl(2|1)$ one has $\nn'=\nn$ since simple $\sl(2)$-modules are self-dual. Thus 
$$
J^{\sl(2|1)}_{L,V(n,a)}(t,q)=J^{\sl(2|1)}_{L,V(n,a)}(q^{-n-1}t^{-1},q).
$$
Indeed, $V_n$ has basis $F_{12}^eF_2^{e'}F_1^iv_0$ for $i=0,...,n$ and $e,e'=0,1$.
The lowest weight vector is $F_{12}F_2F_1^nv_0$ since both $F_1$ and $F_2$ are zero
over this vector as it is easy to check. Clearly
$K_2(F_{12}F_2F_1^nv_0)=q^{n+1}q^a F_{12}F_2F_1^nv_0$ so this 
vector has $K_2$-degree $n+1+a$ (and $K_1$-degree $2-2(n+1)+n=-n$). Thus the dual has 
highest weight $(n,-a-n-1)$. So $\o(V(n,a))=V(n,-a-n-1)$ that is $l_n=-n-1$.
\end{remark}


\section{Proof of the genus bound}
\label{sec.proofs}

\subsection{Algebraic bound}
\label{sub.algb}

In this section we give a series of lemmas that ultimately imply the algebraic bound
of Proposition~\ref{prop.algb} highlighting the importance of the number of positive 
odd roots.
\medskip

Denote 
$$
[K,n]_{q}:=\frac{q^nK-q^{-n}K^{-1}}{q-q^{-1}} 
$$
(see~\cite[Sec.1.3]{Jantzen}).
We denote $\UQ^-[K_i, \Z]_{q_i}$ 
the $\Q(q)$-subspace spanned by vectors of the form $f[K_i,n]_{q_i}$
for some $f\in \UQ^-$ and $n\in \Z$. Since
$K_iF^J=q^{-(d(J), \a_i)}F^JK_i$, it follows that 
$$
[K_i,\Z]_{q_i}\UQ^-= \UQ^-[K_i,\Z]_{q_i} \,.
$$
The following is the analogue of \cite[Lemma 2.8]{LNV:genus-unrolled}
(but in a simpler form). 

\def\rk{\text{rank}}
\begin{lemma}
\label{lemma: commuting E_i with Fs}
 For every $i=1,\dots,r=\rk(\gg)$ and $f\in \UQ^-$ we have $$[E_i,f]\in \UQ^-[K_i,\Z]_{q_i}$$
\end{lemma}

\begin{proof}
This is clearly true if $f=F_j$ is a generator of $\UQ^-$. If the statement is true for 
$f,f'\in \UQ^-$ then 
\begin{align*}
  [E_i,ff']=[E_i,f]f'+(-1)^{|f||E_i|}f[E_i,f']\in \UQ^-[K_i,\Z]_{q_i}\cdot
  \UQ^-+\UQ^-\cdot \UQ^-[K_i,\Z]_{q_i}
\end{align*}
which is in $\UQ^-[K_i,\Z]_{q_i}$ since $[K_i,\Z]_{q_i}\UQ^-=\UQ^-[K_i,\Z]_{q_i}$.
By induction, this proves the lemma.
\end{proof}

Let $\Bnna$ be a PBW basis of $\Vnna$, i.e., a basis of the form $F^Jv_0$ for some $J:\Phi^+\to \N_{\geq 0}$, for instance, the basis of Lemma \ref{lemma: PBW on V(n,a)}.

\begin{lemma}
\label{lemma: coefficients on V(n,a)}
In the basis $\Bnna$ of $\Vnna$, root vectors act with the following coefficients:
\begin{enumerate}
    \item All negative root vectors act by matrices with 
coefficients in $\Q(q)$.
\item All even positive root vectors act by matrices with 
coefficients in $\Q(q)$.
\item All odd positive root vectors act by matrices with coefficients in
$\Q(q)q^{a}+\Q(q)q^{-a}$.
\end{enumerate}
Moreover, for $i\neq m$, $K_i$ acts with $\Q(q)$-coefficients, while $K_m$ has
coefficients in $\Q(q)q^a.$
\end{lemma}

\begin{proof}
The statement for the $K_i$'s is obvious from the definition of $\Vnna$. Statement
$(1)$ is also clear by the PBW theorem on $\UQ^-$. We now prove $(2)$ and $(3).$
Consider a simple positive root 
vector $E_i$ and let $fv_0$ be a vector in $\Bnna$ with $f\in \UQ^-$. 
By Lemma \ref{lemma: commuting E_i with Fs} we can 
write 
$$
[E_i,f]=\sum_jf_j[K_i, m_j]_{q_i}
$$ for some $f_j\in \UQ^-, m_j\in \Z$. Since $E_iv_0=0$ we get 
$$
E_i(fv_0)=\sum_j f_j\frac{q_i^{m_j}K_i-q_i^{-m_j}K_i^{-1}}{q_i-q_i^{-1}}v_0.
$$
If $E_i$ is even, then $K_iv_0=q^{n_i}v_0$ so that
$$
E_i(fv_0)=\sum_j [m_j+n_i]_{q_i}(f_jv_0)
$$
and if $E_i=E_m$ is odd (so $d_m=1$ and $q_m=q$), then
$$
E_m(fv_0)=\sum_j[m_j+a]_{q}(f_jv_0).
$$

Each $f_jv_0$ can be expanded in the basis $\Bnna$ with $\Q(q)$-coefficients, which
implies that 
$E_i(fv_0)$ has $\Q(q)$-coefficients for even $E_i$ and coefficients in
$\Q(q)q^a+\Q(q)q^{-a}$ for 
$E_m$. This proves the lemma for simple positive roots. Now, if $\a\in \Deev$, 
Lemma \ref{lemma: odd root vectors have a unique Em} says that $E_{\a}$ can be
expanded as 
a sum of products of positive even simple root vectors, hence $E_{\a}$ is represented as 
a sum of products of matrices with $\Q(q)$-coefficients, hence 
so does $E_{\a}$. If $\a\in\Deodd$, Lemma \ref{lemma: odd root vectors have a unique Em}
says that $E_{\a}$ is a sum of products of simple roots where each product has a
unique odd simple root. The lemma follows again from the statement on simple
positive roots.
\end{proof}

Now let 
\be
\label{Xdef}
X=\Vnna\ot \Vnna^*
\ee
equipped with the basis $\BB$ which is the basis $\Bnna$ tensored with its dual basis.

\begin{lemma}
\label{lemma: coefficients on X}
In the basis $\BB$ of $X$, root vectors act with the following coefficients:
\begin{enumerate}
\item All negative root vectors act with coefficients in $\Q(q)$.
\item All even positive root vectors act with coefficients in $\Q(q)$.
\item All odd positive root vectors act with coefficients in $\Q(q)+\Q(q)q^{-2a}$.
\end{enumerate}
Moreover, all Cartan vectors act with $\Q(q)$-coefficients.
\end{lemma}

\begin{proof}
That the Cartan vectors act with $\Q(q)$-coefficients is easy to see since 
$$
K_m(b_i\ot b_j)=q^{l_i+a-l_j-a}(b_i\ot b_j)=q^{l_i-l_j}b_i\ot b_j
$$
if $K_mb_i=q^{l_i+a}$ on a basis vector $b_i$ of $\Bnna$. All other $K_i$'s act with 
$\Q(q)$-coefficients on $\Vnna$, hence on $X$. Thus, the generators of $\UQ^0$ act with 
$\Q(q)$-coefficients, hence all vectors in $\UQ^0$ do.
\medskip

We now prove $(1)$. The simple negative root vectors act on $X$ by 
$\De(F_i)=K_i^{-1}\ot F_i+F_i\ot 1$. Clearly, $F_i\ot 1$ acts on $X$ with
$\Q(q)$-coefficients by Lemma \ref{lemma: coefficients on V(n,a)} so let's see the
coefficients of $K_i^{-1}\ot F_i$ on $\BB$. By definition of the dual action, $F_i$
acts on the dual basis of $\Vnna^*$ as the transpose of the matrix representing
$S(F_i)=-K_iF_i$ over $\Bnna$. When $F_i$ is even, this matrix has $\Q(q)$-coefficients
(since $K_i$ and $F_i$ do), hence so does $K_i^{-1}\ot F_i$. When $F_i=F_m$ is odd,
the same matrix has coefficients in $\Q(q)q^a$, 
but this extra $q^a$ cancels with the $q^{-a}$ from the action of $K_m^{-1}$ on the first 
tensor factor. In all cases, $F_i$ acts with $\Q(q)$-coefficients in the basis $\BB$
of $X$. Having proved it for the generators of  $\UQ^-$, the statement follows over
all $\UQ^-$.
\medskip

We now prove $(2)$ and $(3)$. Here we use that $\De(E_i)=E_i\ot K_i+1\ot E_i$. When
$E_i$ is even, it is clear this has $\Q(q)$-coefficients on $\BB$ since $E_i,K_i$
have $\Q(q)$-coefficients on $\Bnna$. When $E_i=E_m$ is odd, the coefficients of
$E_m\ot K_m$ are in 
$$
(\Q(q)q^a+\Q(q)q^{-a})\Q(q)q^{-a}=\Q(q)+\Q(q)q^{-2a}
$$
by $(3)$ of Lemma \ref{lemma: coefficients on V(n,a)}. The coefficients of $1\ot E_m$
have the same form since $E_m$ acts on $\Vnna^*$ as the transpose of $-E_mK_m^{-1}$.
This proves the statement for the positive simple roots. The statement for all
positive roots follows from Lemma \ref{lemma: odd root vectors have a unique Em} as
in the proof of the previous lemma.
\end{proof}

The following is the analogue of \cite[Prop. 2.11]{LNV:genus-unrolled}. Denote by 
$U_{\Q}$  the $\Q(q)$-subalgebra of $U_q(\gg)$ generated by the $K_i^{\pm 1},
E_i,F_i$ for $i=1,\dots,r$. 

\begin{proposition}(Algebraic bound)
\label{prop.algb}
Let $Z\in U_{\Q}$. Then, in the basis $\BB$ of $X$, the action of $Z$ on $X$ 
is represented by a matrix whose coefficients are polynomials in $q^{-2a}$ 
(with $\Q(q)$-coefficients) of degree $\leq \varphi^+=|\Deodd|$ in $q^{-2a}$.
\end{proposition}

\begin{proof}
We first prove it for a vector of the form $Z=E^J$ for some $J:\Phi^+\to \N_{\geq 0}$,
that is $Z=\prod_{\a\in\Phi^+}E_{\a}^{J(\a)}$ where $J(\a)\in \{0,1\}$ for every odd
root $\a$. By the previous lemma, for each even $\a$, $E_{\a}^{J(\a)}$ is represented
by a matrix with $\Q(q)$-coefficients. For each odd $\a$, $E_{\a}^{J(\a)}$ is
represented by a matrix with coefficients of degree one in $q^{-2a}$ if $J(\a)=1$
(or the identity matrix if $J(\a)=0$). Multiplying all these matrices, we get 
that $Z=E^J$ is represented by a matrix with coefficients of the desired form.
\medskip

Now write an arbitrary $Z$ in the PBW basis as $Z=\sum c_{JAI}(q) E^JK^AF^I$ where 
$c_{JAI}(q)\in \Q(q)$. Since each $K^AF^I$ acts on $\BB$ with $\Q(q)$-coefficients,
each $E^J$ satisfies the property of the lemma, and the property is preserved after
taking sums\footnote{Here 
it is important that the coefficients of $E^J$'s are polynomials in $q^{-2a}$ of degree 
$\leq \phip$ instead of merely Laurent polynomials of span $\leq \phip$.}, 
the statement follows for $Z$.
\end{proof}

Recall that $X=V\ot V^*$ where $V=V(\nn,a)$. Then $X^*\cong V^{**}\ot V^*\cong X$ via 
the pivotal structure of $\CC$, and recall from Section~\ref{sub.piv} the corresponding
isomorphism 
\be
\label{jX}
\iota_X:X^*\to X, \qquad \iota_X=(j_V\ot\id_{V^*})\circ \psi
\ee
where $\psi$ is the canonical isomorphism $X^*\cong V^{**}\ot V^*$. 

\begin{lemma}
\label{lem.jX}
$\iota_X$ is represented by a diagonal matrix (in the standard basis) with
coefficients in $q^{a\varphi^+}\Q(q)$.
\end{lemma}

\begin{proof}
Since every odd positive root can be written as $\a=\b+\a_m$ with $\b\in \Deev$ it
follows that $K_{2\rhoodd}:\Vnna\to \Vnna$ has coefficients in $q^{a\varphi^+}\Q(q)$.
Thus, since $j_V$ is given by~\eqref{jVdef} and because of~\eqref{rho} and Lemma 
\ref{lemma: coefficients on V(n,a)}, it is represented by a diagonal matrix 
(in the usual basis) with coefficients in $q^{a\varphi^+}\Q(q)$. This and~\eqref{jX}
imply that the same holds for $\iota_X$.
\end{proof}

\subsection{Seifert surface argument}
\label{subs: Seifert formula}

Let $L$ be an unframed, oriented link and let $\S$ be a connected Seifert surface for 
$L$. Let $N=\text{rank}_{\Z} H_1(\S;\Z)$ and $g=g(\S)$, so $N=2g+s-1$ if $L$ has $s$
components. After an isotopy, we can suppose that $\S$ is obtained by thickening a
framed tangle $\TS\sb \R^2\t[0,1]$ with $N$ components, where each component is a
band with both endpoints on a 
fixed line in $\R^2\t\{1\}$ and then attaching a disk on top. For simplicity of
notation, we will 
suppose the first $2g$ components of $\TS$ alternate their endpoints, and the last
$s-1$ components 
have their endpoints next to each other (this can always be assumed after an isotopy
of $\S$). 
Let $L_o$ be the framed $(1,1)$-tangle obtained by opening 
$L=\partial\S$ at a point in the top rightmost corner of the disk attached, see Figure 
\ref{figure:Seifert}. We will consider $L_o$ as a framed tangle with the 
framing coming from $\S$, thus $w(L_o)=0$ and the quantum invariants need no framing
normalization. 
In what follows, by standard basis of $V=\Vnna, V^*, X$ or $X^*$ we mean a PBW basis of
$V$ (say, that of Lemma \ref{lemma: PBW on V(n,a)}), its dual 
basis, the tensor product of these two basis, or the dual of the latter, respectively.

\begin{figure}[H]
\centering
\includegraphics[width=14cm]{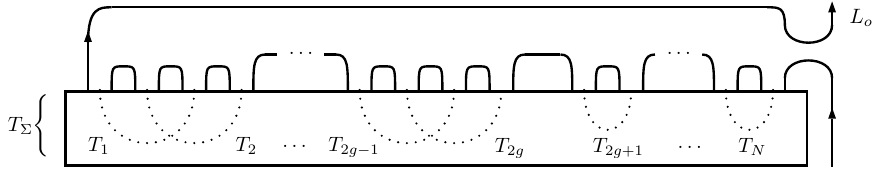}
\label{figure:Seifert}
\caption{The tangle $L_o$ obtained by cutting $L=\partial\S$ at a point. The
  tangle $\TS$ lies inside the box, in this box, all components $T_1,\dots,T_N$
  might be linked together. }
\end{figure}

Recall that $X=V\ot V^*$ where $V=V(\nn,a)$.
The Reshetikhin--Turaev invariant $F(L_o,V): V \to V$ is computed in three steps.

First, suppose each component of $\TS$ is oriented to the right, that is, for any
component $T_i$ of $\TS$ the orientation of $T_i$ is downwards at the leftmost endpoint
of $T_i$ on $\R^2\t\{1\}$ and upwards at the rightmost endpoint. Then the
Reshetikhin--Turaev invariant of the tangle $\TS$ colored with $X$ is a morphism 
$$
F(\TS,X):\C\to  (X^*\ot X^*\ot X\ot X)^{\ot g}\ot (X^*\ot X)^{\ot s-1}.
$$
Second, 
let $\FF(T,X)$ be the result of applying $\iota_X:X^*\to X$ to each $X^*$ tensor factor:
\begin{align}
\label{eq: Fbar from F and j's}
\FF(\TS,X):=[(\iota_X\ot \iota_X\ot \id_X\ot \id_X)^{\ot g}\ot 
(\iota_X\ot\id_X)^{\ot s-1}]\circ F(\TS,X):\C\to X^{\ot 2N}.
\end{align}
Finally $F(L_o,V):V\to V$ is given by
\begin{align}
\label{eq: LG from FbarTX}
F(L_o,V)=(\id_{V}\ot \lev_{V}^{\ot 2N})\circ(\FF(\TS,X)\ot\id_{V}).
\end{align}

\begin{lemma}
\label{lemma: braiding universal inv form on XotX}
The braiding on $X\ot X$ can be written as 
$$
c_{X,X}(v\ot w)=\tau_{X,X}\left(R_X(v\ot w)\right)
$$
for some $R_X\in U_{\Q}^{\ot 2}$. Similarly, the inverse can be written as 
$c_{X,X}^{-1}(v\ot w)=\tau(\overline{R}_X(v\ot w))$ for some 
$\overline{R}_X\in U_{\Q}^{\ot 2}$. Here $\tau_{X,X}$ 
denotes the usual transposition of vector spaces: $\tau_{X,X}(u\ot v)=v\ot u$.
\end{lemma}

\begin{proof}
From the formula ~\eqref{cuv} of the braiding and the definition of $\HH$
from~\eqref{Hpart} we just need to prove that $\HH:X\ot X\to X\ot X$ can be 
written as the action of some element in $U_{\Q}^{\ot 2}$. In other words, we need to
show there exist coefficients $h_{\c\d}\in \Q(q)$ with $\c,\d$ in the root lattice
$L_R$ such that 
$$
q^{(\a,\b)}x\ot y=\HH(x\ot y)=\sum_{\c,\d\in L_R}h_{\c\d}K_{\c}x\ot K_{\d}y
$$
for all $x,y\in X$ of weights $\a,\b\in L_R$ respectively, where
$K_{\c}=\prod_{i=1}^rK_i^{c_i}$ 
if $\c=\sum c_i\a_i$ ($c_i\in \Z$). This is equivalent to finding $h_{\c\d}$ such that
$$
q^{(\a,\b)}=\sum_{\c,\d\in L_R}h_{\c\d}q^{(\a,\c)+(\b,\d)} \,.
$$
for all weights $\a,\b\in L_R$ of $X$ (which is a finite set). Expressing
$\a,\b,\c,\d$ in the basis of simple roots, we are led to proving that one can write 
$$
q^{a a_{ij} b}=\sum_{c,d\in \Z}h_{cd}q^{ac+bd}.
$$
for all $a,b$ in a finite set of integers. This is a system of linear equations in the 
variables $h_{cd}$ of Vandermonde type, which can be shown to always have a solution when 
$q$ is generic. Since the quasi-R-matix $\Theta$ is invertible in $U_{\Q}^{\ot 2}$ the 
statement for the inverse follows from what we just proved.
\end{proof}

Thanks to the previous lemma, despite not having a universal invariant in 
$U_{\Q}^{\ot N}$ for 
$\TS$ that works for all representations (because the operator $q^{H\ot H/2}$ cannot
be written as 
$\sum a\ot b$ with $a,b\in U_{\Q}$ acting on all representations), there 
is such an element that works for $X$-colored tangles:

\begin{lemma}
There exists an element $Z_{\TS}\in U_{\Q}^{\ot N}$ (that depends on $X$) such that 
$\FF(\TS,X)$ is left multiplication by $Z_{\TS}$. More precisely, 
\begin{align}
\label{eq: FTX from ZT}
F(\TS,X)=P\circ L_{Z^{ev}_{\TS}}\circ  \coev_X^{\ot N}
\end{align}
   
\noindent 
where $P$ is a permutation isomorphism of super-vector spaces and $L_{Z^{ev}_{\TS}}$ means 
left multiplication by $Z_{\TS}$ only on the even tensor factors of $X^{\ot 2N}$.
\end{lemma}

We can be more explicit about the isomorphism $P$ and say that 
$$
P=(\id_{X^*}\ot \tau_{X,X^*}\ot \id_X)^{\ot g}\ot \id_{X^*\ot X}^{\ot (s-1)}
$$ 
but we won't need this. As an example, if $g(\S)=1$ and $s=1$ (so $L$ is a knot) the
formula of the lemma states that 
$$
F(\TS,X)=\sum_{i,j}e'_i\ot e'_j\ot Z_1\cdot e_i\ot Z_2\cdot e_j
$$
if $Z_{\TS}=\sum Z_1\ot Z_2$ and $(e_i)$ is any basis of $X$ with dual basis $(e'_i)$.

\begin{proof}
This follows from the usual proof of existence of the universal invariant thanks to Lemma 
\ref{lemma: braiding universal inv form on XotX}; see also~\cite{Ohtsuki:univ}.
Indeed, after writing the braiding $c_{X,X}$ (and its inverse) in the form of Lemma 
\ref{lemma: braiding universal inv form on XotX}, $Z_{\TS}$ is 
obtained simply by following each of the $N$ strands of $\TS$ and multiplying all
the elements encountered (assuming the tangle has been isotoped so that all crossings
are oriented upwards). All coefficients are in $\Q(q)$ since this is the case for the
$R$-matrix.
\end{proof}

\subsection{Proof of Theorem~\ref{thm.1.links} for connected $\S$}

We now have all the ingredients to prove Theorem~\ref{thm.1.links} for connected $\S$.
We suppose first that $\S$ is connected as in the previous discussion. From formula 
\eqref{eq: FTX from ZT} and the algebraic bound of Proposition \ref{prop.algb} (applied 
once for every component of $\TS$), it follows that when $F(\TS,X)$ is expanded in
the standard basis of $(X^*\ot X^*\ot X\ot X)^{\ot g}\ot (X^*\ot X)^{s-1}$, all
coefficients that appear will be polynomials in $q^{-2a}$ (with $\Q(q)$-coefficients)
of degree $\leq N\cdot \varphi^+$. Now, since $\iota_X:X^*\to X$ is represented in
the standard bases by a (diagonal) matrix with coefficients in $q^{a\varphi^+}\Q(q)$,
it follows from (\ref{eq: Fbar from F and j's}) 
(which involves $N$ $\iota_X$'s) that $\FF(\TS,X)$ has coefficients which are Laurent
polynomials in $q^{a}$ with powers lying in the interval $[-N\varphi^+, N\varphi^+]$. 
Finally, since the left evaluation $\lev_V$ is $0$ or $1$ in the standard basis (in
particular it involves no powers of $q^a$), if follows from (\ref{eq: LG from FbarTX})
that $F(L_o,V)$ on any standard basis vector, is a multiple of that vector by a
polynomial in $q^a$ with powers in $[-N\varphi^+, N\varphi^+]$. The latter polynomial
is the invariant $J^{\gg}_{K,V}$, 
therefore proving the theorem in the case of connected $\S$.
\qed

\begin{remark}
The above also proves that the invariants $J_{K,V}^{\gg}$ are Laurent polynomials in
$q^{2a}$ up to an overall power of $q^a$. This overall power is
$q^{a\varphi^+(2g(L)+s-1)}$. In particular, the 0-framed knot invariants are always
Laurent polynomials in $q^{2a}$. Compare also with Remark~\ref{rem.int}.
\end{remark}

\subsection{The case of disconnected $\S$} 

Suppose now that $\S$ has $k$ components. Then it can be obtained from a bottom tangle 
$\TS$ by attaching $k$ disks on top as indicated in Figure
\ref{fig:Seifert surface disconnected case}. To get the open link $L_o$ we cut open
the rightmost disk, as indicated in the same figure.

\begin{figure}[h]
    \centering
    \includegraphics[width=10cm]{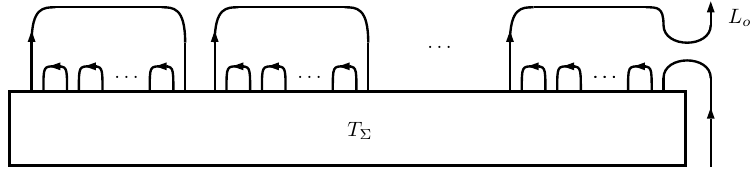}
\caption{This surface is obtained from $k$ disjoint disks by attaching the tangle $\TS$ 
  from the bottom. Note that only one of the disks (the rightmost in the above figure)
  has to be cut open.}
\label{fig:Seifert surface disconnected case}
\end{figure}

Now comes the crucial observation: each of the $k-1$ disks (those that were not cut
open) contributes a right evaluation $\rev_V$ to the formulas. This right evaluation
has coefficients in $q^{-a\varphi^+}\Q(q)$ in the standard basis. The Reshetikhin--Turaev
invariant $\FF(\TS,X)$ has again coefficients that are polynomials in $q^{a}$ with
powers lying in $[-N\varphi^+,N\varphi^+]$ where $N$ is the number of components of
$\TS$. Since left evaluations contribute no $q^a$'s, it follows (as in the previous
proof) that the resulting invariant $J_{L,V(\nn,a)}$ has powers of $q^a$ 
lying in 
$$
[-N\varphi^+-(k-1)\varphi^+,N\varphi^+-(k-1)\varphi^+].
$$
Note that the latter fact is independent of the color $\nn$ chosen, hence the same
applies as well to the invariant $J_{L,V(\nn',a)}$ where $\nn'$ is such that
$\o(V(\nn,a))\cong V(\nn',-a+l)$ ($l\in \Z$). But by Proposition
\ref{prop: symmetry property}, the invariants $J_{L,V(\nn,a)}$ 
and $J_{L,V(\nn',a)}$ are related by $q^a\mapsto q^{-a}q^l$. It follows that the
powers of $q^a$ for both must lie in the (smaller) interval 
$$
[-N\varphi^++(k-1)\varphi^+,N\varphi^+-(k-1)\varphi^+].
$$
But $N-(k-1)=1-\chi(\S)$ which proves that $\deg_{q^{2a}}\leq (1-\chi(\S))\varphi^+$ 
as desired.
\qed


\section{The MMR Conjecture for Lie superalgebras}
\label{sub.MMR}

In this section we sketch the argument of the MMR Conjecture from simple Lie 
algebras to Lie superalgebras.
The proof of the MMR conjecture (given in~\cite{B-NG} for simple Lie algebras and
extended mutatis-mutandis for Lie superalgebras) implies that

\be
J^\gg_{K,\Vnna}(t,q) =
\frac{\prod_{\a \in \Deodd} \Delta(e^{h (\lambda,\a)})}{
  \prod_{\a \in \Deev} \Delta(e^{h (\lambda,\a)})} + O(q-1)
\ee
where $\lambda$ is the weight $(\nn,a)$, $q=e^{h/2}$ (see~\cite[Sec.1.2]{Geer:multi1})
and $t=q^{2a}=e^{h a}$. 
For every $\a \in \Phi^+$, we have $(\lambda, \a) = \ve a + L_\a(\nn)$ 
where $\ve=0$ (resp., $1$) if $\a$ is even (resp., odd) and $L_\a(\nn)$ is a 
linear form on $\nn$. It follows that for every fixed $\nn$, we have
\be
J^{\gg}_{K,\Vnna}(t,1) = \Delta_K(t)^{\phip} \,. 
\ee


\bibliographystyle{hamsalpha}
\bibliography{biblio}
\end{document}